\documentclass[a4paper,twoside]{article}      % Comments after  % are ignored
\usepackage{amsmath,amssymb,amsfonts,amsthm,amscd} % Typical maths resource packages
\usepackage{graphics}                 % Packages to allow inclusion of graphics
\usepackage{color}                    % For creating coloured text and background
\usepackage{ mathrsfs }
\usepackage{indentfirst}
\usepackage{bbold}
\usepackage{enumerate}
\usepackage{url}         %to include url's in citations
\usepackage{colonequals} %to use the := symbol
\usepackage{a4wide}

\newtheorem{theorem}{Theorem}[section]

\newtheorem{lemma}[theorem]{Lemma}
\theoremstyle{definition}
\newtheorem{remark}[theorem]{Remark}

\newtheorem{definition}[theorem]{Definition}

\def\div{\mathop{\mathrm{div}}\nolimits}
\def\!{\mathop{\mathrm{!}}}

\def\Prob{\mathop{\mathrm{Prob}}}
\def\grad{\mathop{\mathrm{grad}}}

\newcommand{\Keywords}[1]{\par\indent
{\small{\textbf{Key words and phrases.} \/} #1}}

\def\R{\mathbf{ R}}

\def\J{\mathbf{J}}
\def\P{\mathcal{P}}

\def\D{\mathcal{D}}

\def\M{\mathcal{M}}
\def\T{\mathcal{T}}

\def\aL{\mathsf L}
\def\aM{\mathsf M}
\def\aE{\mathsf E}
\def\aS{\mathsf S}
\def\aZ{\mathsf Z}
\def\az{\mathsf z}
\def\aF{\mathsf F}
\def\aG{\mathsf G}
\def\ad{{\mathsf d}}
\def\mS{\mathcal S}
\def\mE{\mathcal E}
\def\mH{\mathcal H}
\def\cD{\mathcal D}

\def\dual#1#2{\left\langle#1,#2\right\rangle}
\DeclareMathOperator\supp{supp}
\DeclareMathOperator\id{id}

\def\Expectation{\mathbb{E}}
\DeclareMathOperator\law{law}

\newlength{\boxwidth}
\title{GENERIC formalism of a Vlasov-Fokker-Planck equation and connection to  large-deviation principles}
\author{Manh Hong Duong$^1$ \and Mark A. Peletier$^{1,2}$ \and Johannes Zimmer$^3$}
\date{\today}
\begin{document}
\maketitle

\begin{abstract}
In this paper we discuss the connections between a Vlasov-Fokker-Planck equation and an underlying microscopic particle system, and we interpret those connections in the context of the GENERIC framework (\"Ottinger 2005). This interpretation provides (a) a variational formulation for GENERIC systems, (b) insight into the origin of this variational formulation, and (c) an explanation of the origins of the conditions that GENERIC places on its constitutive elements, notably the so-called degeneracy or non-interaction conditions. This work shows how the general connection between large-deviation principles on one hand and gradient-flow structures on the other hand extends to non-reversible particle systems.
\end{abstract}
\footnotetext[1]{Department of Mathematics and Computer Sciences, Technische Universiteit Eindhoven}
\footnotetext[2]{Institute for Complex Molecular Systems, Technische Universiteit Eindhoven}
\footnotetext[3]{Department of Mathematical Sciences, University of Bath}

\section{Introduction}
\label{sec:intro}

\subsection{Overview}
The framework GENERIC~\cite{Oettinger05} provides a systematic method to derive thermodynamically consistent evolution equations. It was originally introduced in the context of complex fluids~\cite{OG97,OG97part2}, and more recently has been applied to anisotropic inelastic solids~\cite{HutterTervoort08a}, to viscoplastic solids~\cite{HutterTervoort08b}, and thermoelastic dissipative materials~\cite{Mielke11}. The key ingredients of GENERIC are its building blocks: a Poisson operator $\aL$, a dissipative operator $\aM$, an energy functional $\aE$, and an entropy functional $\aS$, which are required to satisfy certain properties. Although many equations have been shown to have a GENERIC structure, two important aspects have not been addressed.

The first one is the relationship between the GENERIC framework on one hand and large deviations of underlying microscopic particle systems on the other. It is well-known that many deterministic evolution equations can be derived as hydrodynamic limits of a stochastic particle system. More recently it has become clear the connection between  particle systems and their upscaled limits runs deeper: gradient-flow structures of the limit equations arise as characterizations of the large-deviation behaviour of the stochastic particle systems, thus explaining amongst other things the origin of the Wasserstein gradient flows~\cite{ADPZ11,ADPZ12,Laschos2012,Peletier2011,Renger13TH}. In this paper we generalize this relationship beyond gradient flows to an example from the class of GENERIC systems.

The second aspect is a variational structure for GENERIC systems. The study of variational structure has important consequences for the analysis of an evolution equation. It provides general methods for proving well-posedness~\cite{AGS08} and characterizing large-time behaviour~(e.g.,~\cite{CarrilloMcCannVillani03}), gives rise to natural numerical discretizations (e.g.,~\cite{DuringMatthesMilisic10}), and creates handles for the analysis of singular limits (e.g.,~\cite{SandierSerfaty04,Stefanelli08,AMPSV12}). The appearance of the concepts of energy and entropy in the formulation of GENERIC suggests a strong variational connection, but to date this has not been made explicit. In this paper we exhibit such a variational structure, and as in the case of the gradient flows, this structure is intimately tied to the large-deviation behaviour of an underlying system.

In this paper we treat some of these questions in full generality, that is, for a general, abstract GENERIC system. Because of this generality the treatment is necessarily formal. We illustrate the abstract features with a specific system, that of the \emph{Vlasov-Fokker-Planck equation}, for which the large-deviation behaviour has been proved rigorously. This gives a specific case in which the impact of the abstract arguments can be recognized. We first introduce  the specific example and then explain the GENERIC framework in detail.

\subsection{A Vlasov-Fokker-Planck equation and its generalisation}
\label{subsec:VFP}

The central example of this paper will be the following \emph{Vlasov-Fokker-Planck (VFP)} equation,
\begin{equation}
\label{eq:VFP}
\partial_t \rho = - \div_q\Bigl(\rho \frac pm\Bigr) + \div_p  \rho \Bigl(\nabla_q V + \nabla_q \psi * \rho + \gamma\frac pm \Bigr)+\gamma \theta\Delta_p \rho.
\end{equation}
The spatial domain is $\R^{2d}$ with coordinates $(q,p)$, with $q$ and $p$ each in $\R^d$. We use subscripts as in $\div_q$ and $\Delta_p$ to indicate that the differential operators act only on those variables. The unknown is a time-dependent probability measure $\rho\colon[0,T]\to \P(\R^{2d})$; the functions $V = V(q)$ and $\psi = \psi(q)$ are given, as are the positive constants $\gamma$, $m$, and $\theta$. The convolution $\psi * \rho$ is defined by $(\psi\ast\rho)(q)=\int_{\R^{2d}}\psi(q-q')\rho(q',p')\,dq'dp'$.

Equation~\eqref{eq:VFP} arises as the many particle limit of a collection of \emph{interacting Brownian particles with inertia}~\cite{Brown}, given by the following stochastic differential equation
\begin{subequations}
\label{eq:SDE}
\begin{align}
dQ_i(t)&=\frac{P_i(t)}{m}\,dt,
\\dP_i(t)&=-\nabla V(Q_i(t))\,dt - \sum_{\substack{j=1}}^n \nabla \psi(Q_i(t)-Q_j(t))-\frac{\gamma}{m}P_i(t)\,dt+\sqrt{2\gamma \theta }\,dW_i(t).
\end{align}
\end{subequations}
Here $Q_i$ and $P_i$ are the position and momentum of particle $i=1,\dots,n$, with mass $m$, and the equations describe the movement of this particle under a fixed potential $V$, an interaction potential $\psi$,  a friction force (the drift term $-\gamma P_idt/m$) and a stochastic forcing described by the $n$ independent $d$-dimensional Wiener measures~$W_i$.

Both the friction force and the noise term arise from collisions with the solvent, and the parameter~$\gamma$ in both terms characterizes the intensity of these collisions. The parameter $\theta=kT_a$, where $k$ is the Boltzmann constant and $T_a$ is the absolute temperature, measures the mean kinetic energy of the solvent molecules, and therefore characterizes the magnitude of the collision noise.  Typical applications of this system are for instance  as a simplified model for chemical reactions, or as a model for particles interacting through Coulomb or gravitational forces.

Equation~\eqref{eq:VFP} is the many limit of the SDE~\eqref{eq:SDE}, also known as the hydrodynamic limit, in the sense that as $n\to\infty$, the \emph{empirical measure}
\[
\rho_n(t)\colonequals \frac1n \sum_{i=1}^n\delta_{(Q_i(t),P_i(t))}
\]
converges almost surely to the solution of~\eqref{eq:VFP} with appropriate initial data. Equation~\eqref{eq:VFP} has been extensively studied, especially in the case in which $\psi$ is a Coulomb or gravitational potential. The central difficulty in these works is the singularity of $\psi$. For our purposes, this issue is not important, and we will simply assume that $\psi$ is bounded, thus eliminating  difficulties in proving existence and uniqueness.

% Earlier work:
% Chandrasekhar 43: for V(q) = q^2, general background on Brownian motion
% Neunzert-Pulvirenti-Triolo 84: Poisson, no background potential
% Victory-O'Dwyer 90: no background potential, only Poisson interaction
% Carrillo-Soler 97: existence for measure initial data

\bigskip

Although we prove a rigorous result, Theorem~\ref{theo:LDP}, the main statement of this paper is not Theorem~\ref{theo:LDP}; the main statement is the general structure that Theorem~\ref{theo:LDP} strongly suggests, which extends much further than the example above, and which connects to the GENERIC structure that we describe below. Because of this suggestion of a general structure, we now describe a generalized version of the Vlasov-Fokker-Planck equation in somewhat more abstract terms.

\bigskip

Let $\mH,\mS\colon \P(\R^{2d})\rightarrow \R$ be two functionals on $\P(\R^{2d})$. Denote by $\grad \mH$ and $\grad \mS$ the $L^2$-gradient of $\mH$ and $\mS$, otherwise known as the variational derivative.

The following equation we call a \emph{generalized Vlasov-Fokker-Planck equation},
\begin{equation}
\label{eq:generalizedKramers}
\partial_t\rho=\div(\rho \J\nabla \grad \mH) +\div\bigl(\rho\sigma\sigma^T\nabla \grad (\mH+\mS)\bigr),
\end{equation}
where $\nabla$ and $\div$ are the gradient and divergence operators with respect to the full spatial variable ${x} = (q,p)\in \R^{2d}$, and $\J$ is the $2d\times 2d$ skew symmetric block matrix
\begin{equation}
\label{def:symplecticMatrixJ}
\J =\begin{pmatrix}
0 &-I_d \\
I_d& 0
\end{pmatrix},
\end{equation}
where $I_d$ is the $\R^{d\times d}-$identity matrix.

The Vlasov-Fokker-Planck equation~\eqref{eq:VFP} is an example of this abstract equation, in which
\begin{subequations}
\begin{alignat}2
\label{def:VFPMeasureStructure.1}
&{x}=(q,p)^T\in \R^{2d},&\quad&\mH(\rho)=\int_{\R^{2d}}\left(\frac{p^2}{2m}+V(q)+\frac{1}{2}(\psi\ast\rho)(q)\right)\, \rho(dqdp),\\
&\sigma=\sqrt{\gamma}
\begin{pmatrix}
0&0\\
0&I_d
\end{pmatrix}, &\qquad& \mS(\rho)=\theta\int_{\R^{2d}}\rho\log\rho\,dqdp.
\label{def:VFPMeasureStructure.2}
\end{alignat}
\end{subequations}
Other well-known equations are of the same form; the \emph{Kramers equation}~\cite{Kramers40} is equation~\eqref{eq:VFP} with $\psi\equiv0$, and Wasserstein gradient flows~\cite{Otto01} are of the form~\eqref{eq:generalizedKramers} with $\sigma = I_{2d}$ and $\mH = 0$.
As a final example, when
\[
\sigma=I_{2d},\quad \mE(\rho)=\frac{1}{2}\int_{\R^{2d}} \rho (\psi\ast\rho), \quad \mS(\rho)=\theta\int_{\R^{2d}}\rho\log \rho,
\]
equation~\eqref{eq:generalizedKramers} becomes
\[
\partial_t\rho=\div(\rho \J \nabla\psi\ast\rho)+ \theta\Delta\rho+\div(\rho\nabla\psi\ast\rho).
\]
This equation describes the relaxation of a point vortex towards statistical equilibrium, that arises in the kinetic theory of point vortices. It is closely related to the two-dimensional Navier-Stokes equation \cite{Chav01,CPR08,CPR09, FSS12}.

\subsection{GENERIC}
\label{subsec:GENERIC}
We now switch gears and introduce an abstract equation structure. Later we will connect the example above with this structure.

A GENERIC equation (General Equation for Non-Equilibrium Reversible-Irreversible Coupling~\cite{Oettinger05}) for an unknown $\az$ in a state space $\aZ$ is a mixture of both reversible  and dissipative dynamics:
\begin{equation}
\label{eq:GENERICeqn1}
\partial_t\az=\aL\,\ad \aE+\aM\,\ad \aS.
\end{equation}
Here
\begin{itemize}
\item $\aE, \aS\colon \aZ\rightarrow\R$ are interpreted as energy and entropy functionals,
\item $\ad \aE, \ad \aS$ are appropriate derivatives of $\aE$ and $\aS$ (such as either the Fr\'echet derivative or a gradient with respect to some inner product);
\item $\aL= \aL(\az)$ is for each $\az$ an antisymmetric operator satisfying the Jacobi identity
\begin{equation}
\label{def:Jacobi}
\{\{\aF_1,\aF_2\}_{\aL},\aF_3\}_{\aL}+\{\{\aF_2,\aF_3\}_{\aL},\aF_1\}_{\aL}+\{\{\aF_3,\aF_1\}_{\aL},\aF_2\}_{\aL}=0,
\end{equation}
for all functions $\aF_i\colon\aZ\rightarrow\R,\ i=1,2,3$,  where the Poisson bracket $\{\cdot,\cdot\}_{\aL}$ is defined via
\begin{equation}
\label{def:bracket}
\{\aF,\aG\}_{\aL}\colonequals\ad \aF\cdot \aL\,\ad \aG
\end{equation}
(see Remark~\ref{rem:differentiability} for a discussion of the meaning of the `dot' here).
\item $\aM=\aM(\az)$ is  symmetric and positive semidefinite.
\end{itemize}
Moreover, the building blocks $\{\aL,\aM,\aE,\aS\}$ are required to fulfill the \emph{degeneracy conditions}: for all $\az\in \aZ$,
\begin{equation}
\label{def:degeneracy}
\aL\,\ad\aS=0,\quad
\aM\,\ad\aE=0.
\end{equation}
As a consequence of these properties, energy is conserved along a solution, and entropy is non-decreasing:
\begin{align*}
\frac{d\aE(\az(t))}{dt}&=\ad \aE\cdot \frac{d\az}{dt}=\ad \aE\cdot\left(\aL\,\ad \aE+\aM\,\ad \aS\right)=0,
\\\frac{d\aS(\az(t))}{dt}&=\ad \aS\cdot \frac{d\az}{dt}=\ad \aS\cdot\left(\aL\,\ad \aE+\aM\,\ad \aS\right)=\ad \aS\cdot\aM\,\ad \aS\geq0.
\end{align*}
A GENERIC system is then fully characterized by $\{\aZ,\aE,\aS,\aL,\aM\}$.

\begin{remark}
\label{rem:differentiability}
In equation~\eqref{eq:GENERICeqn1} we implicitly have assumed that $\aZ$ is a space with a differentiable structure, in which time derivatives $\partial _t\az$ and state-space derivatives $\ad \aS$ and $\ad \aE$ exist. In many cases of importance, including the main example of this paper,  this is not true, and then generalizations are necessary; the book by Ambrosio, Gigli and Savar\'e~\cite{AGS08} is an example of such generalizations in the case of gradient flows. Nonetheless, we feel that the formal differentiable way of writing provides the right intuition, and therefore in this formal part of the paper we maintain this way of writing  the system.

Even in the smooth setting, we  have not made specific exactly which derivative $\ad \aE$ and $\ad \aS$ should be, and let us briefly make the situation concrete. Derivatives of the functionals $\aE$ and $\aS$ are naturally defined as covectors, i.e. elements of the cotangent space (they are then called differentials) or dual space (called Fr\'echet derivatives). Since $\partial_t\az$ is an element of the tangent or primal space, $\aL$ and $\aM$ should be duality maps, mapping cotangent to tangent spaces, or equivalently dual to primal spaces. In this case the meaning of the dot in~\eqref{def:bracket} is that of the duality pairing.

In practice, however, it often is more convenient to use gradients rather than differentials: then the covectorial derivative is mapped to a tangent vector by some fixed duality mapping, associated with an inner product, often only formally. In all of the explicit calculations in this paper this will be the case; for instance, we already used the $L^2(\R^{2d})$ structure as a formal inner product on the space of measures $\P(\R^{2d})$ to define `$\grad \mH$' in equation~\eqref{eq:generalizedKramers}. In this situation $\aL$ and $\aM$ map vectors to vectors, and the dot in~\eqref{def:bracket} is that of the formal inner product.
\end{remark}

\subsection{Overview}

As described in the introduction, the aim of this paper is twofold: to connect the GENERIC structure with large deviations of stochastic processes, and to construct a useful variational formulation for an abstract GENERIC equation.

In this context, the role of the VFP equation~\eqref{eq:VFP} is that of a guiding example. In brief, the story runs as follows: with some modification,  the VFP equation can be written as a GENERIC system. In addition, the VFP equation has a particle background, and a recent large-deviation result allows us to connect the large deviations of the particle system with the GENERIC structure. Finally, this same connection shows how the VFP equation can be given a variational formulation.

The first part of this story is told in Section~\ref{sec:LDPVFP}, in which we construct a large-deviation principle for the SDE~\eqref{eq:SDE} associated with the VFP equation. Next, in Section~\ref{sec:GENERICVFP} we construct a GENERIC structure for the VFP equation and reformulate the large-deviation rate function in this context. Finally, in Section~\ref{sec:variational} we deduce from the large-deviation result a variational formulation for the VFP equation and more generally for any GENERIC system.

Having connected the GENERIC structure with particle systems and large deviations, in Section~\ref{sec:properties} we use this connection to understand the origin and interpretation of the various properties of GENERIC listed in Section~\ref{subsec:GENERIC}. Section~\ref{sec:generalization} is devoted to the generalization~\eqref{eq:generalizedKramers}. We conclude with some comments.

\section{Main results 1: Large deviations for the VFP equation}
\label{sec:LDPVFP}

For many gradient-flow systems it is now understood that the gradient-flow structure itself arises from the fluctuation behaviour of an underlying stochastic process~\cite{DG89,ADPZ11,ADPZ12,Laschos2012,Peletier2011,DuongLaschosRenger12TR,Renger13TH}. The theory of large deviations allows one to make this statement precise. We now apply the same ideas to the VFP equation.

\medskip

We first specify our conditions on the functions $\psi$ and $V$. Since we are interested in presenting ideas rather than obtaining the most general results, we choose fairly restrictive conditions on $V$ and $\psi$ to eliminate technical complications:
\begin{subequations}
\label{cond:Vpsi}
\begin{align}
&V \in C^2(\R^d) \text{ with globally bounded second derivatives, and } V\geq0;
\\
&\psi \in C^2(\R^d) \text{ with globally bounded first and second derivatives, and }\psi\geq0.
\end{align}
In addition, we assume that the initial datum $\rho^0$ satisfies
\begin{equation}
\rho^0\in \P(\R^{2d}) \text{ with } \mH(\rho^0) < \infty,
\end{equation}
where $\mH$ is defined in~\eqref{def:VFPMeasureStructure.1}.
\end{subequations}
With these assumptions,
\begin{itemize}
\item Given a deterministic starting position, the stochastic differential equation~\eqref{eq:SDE} has strong solutions that are weakly unique (see e.g.~\cite[Chapter~3]{KaratzasShreve91}) and non-explosive (e.g.~\cite{Wu01});
\item The VFP equation~\eqref{eq:VFP} is well-defined in the distributional sense and has a unique distributional solution with initial datum $\rho^0$. We do not know of an explicit reference for this, but the boundedness of $\nabla\psi$ implies that standard arguments apply. For instance, existence of mild solutions, defined by the variation-of-constants formula, can be proved using an explicit fundamental solution~\cite{Chandrasekhar43,VictoryODwyer90} of  the differential operator
\[
L\rho\colonequals\partial_t \rho + \div_q\Bigl(\rho \frac pm\Bigr) -\gamma \theta\Delta_p \rho.
\]
Uniqueness follows from an application of Gronwall's inequality.
\end{itemize}

Given a realization $\{(Q_i,P_i)_{i=1}^n\}$ of the particle system~(\ref{eq:SDE}), we define the \emph{empirical measure}
\[
\rho_n\colon [0,\infty) \to \P(\R^{2d}), \qquad \rho_n(t) := \frac1n \sum_{i=1}^n \delta_{(Q_i,P_i)(t)}.
\]
Theorem~\ref{theo:LDP} below states that the random variable $\rho_n$ satisfies a \emph{large-deviation principle} as $n\to\infty$.

\begin{definition}\textbf{(A large-deviation principle~\cite{dH00,DemboZeitouni98,FK06})}
Let $\M$ be a complete separable metric space and $\{\mu_n\}$ be a sequence of probability measures on $\M$. We say that $\{\mu_n\}$ satisfy a large deviation principle with a rate functional $I\colon \M\rightarrow [0,\infty)$ if
\begin{enumerate}
\item  For each open set $A \subset \M$, $\liminf_{n\rightarrow\infty }\frac{1}{n}\log \mu_n(A)\geq -\inf_{x\in A}I(x)$;
\item  For each closed set $B \subset \M$, $\limsup_{n\rightarrow\infty }\frac{1}{n}\log \mu_n(B)\leq -\inf_{x\in B}I(x)$.
\end{enumerate}
The rate functional $I$ is said to be \emph{good} if its sub-level sets $\left\{x\in \M\big\vert I(x)\leq a\right\}$ are compact for all $a\geq 0$.
Morally, this definition describes the property that
\[
\mu_n (A) \sim \exp\bigl(-n\inf_A I\bigr) \qquad\text{as }n\to\infty.
\]
We refer to~\cite{dH00,DemboZeitouni98,FK06} for more information on large deviation theory.
\end{definition}

For the theorem below we equip $\P(\R^{2d})$ with the weak or narrow topology, generated by the duality with $C_b(\R^{2d})$, so that the space $C([0,T];\P(\R^{2d}))$ consists of narrowly continuous curves in $\P(\R^{2d})$.

Define for $\nu\in\P(\R^{2d})$ the parametrized generator
\begin{align*}
A_\nu &\colon D(A_\nu) \subset C_b(\R^{2d})\to C_b(\R^{2d}),\\
A_\nu f &\colonequals \frac pm \cdot \nabla _p f - \Bigl[\nabla_q V + \nabla_q \psi*\nu + \gamma\frac pm\Bigr]\cdot \nabla_p f +\gamma \theta \Delta_p f.
\end{align*}
Note that equation~\eqref{eq:VFP} can be written in terms of the transpose $A^\tau$ as
\[
\partial_t \rho_t = A_{\rho_t}^\tau \rho_t.
\]

For the formulation of the rate function we will also need the concept of absolute continuity in distributional sense. For a compact set $K\subset \R^{2d}$, the space $\cD_K$ is the set of all $f\in C_c^\infty(\R^{2d})$ with $\supp f\subset K$; the set $\cD$ is the union of all $\cD_K$, with the usual test-function topology.

\begin{definition}
A curve $[0,T]\ni t\mapsto \rho_t \in \P(\R^{2d})$ is called \emph{absolutely continuous in distributional sense} if it has the following property: for each compact $K\subset \R^{2d}$ there exists a neighbourhood $U_K$ of $0$ in $\cD_K$ and an absolutely continuous function $G_K:[0,T]\to\R$ such that
\[
\forall\, 0\leq t_1\leq t_2\leq T, \ \forall f\in U_K: \qquad
|\dual{\rho_{t_1}}f - \dual{\rho_{t_2}}f| \leq |G_K(t_1)-G_K(t_2)|.
\]
The set of all such curves is denoted $AC([0,T];\P(\R^{2d}))$.
If $\rho$ is absolutely continuous, then for almost all $t\in [0,T]$ the time derivative
$\partial_t \rho_t $ exists in $\cD'(\R^{2d})$. The proof of this and other properties of this concept can be found in~\cite[Section~4]{DG87}.
\end{definition}

Finally, we define the norm that will measure the magnitude of fluctuations:
\begin{definition}
Fix $\rho\in \P(\R^{2d})$. For any distribution $\T\in \cD'(\R^{2d})$ define
\begin{equation}
\label{def:minusonenorm}
\|\T\|_{-1,\rho}^2 := \sup_{f\in C_c^\infty(\R^{2d})} 2 \dual \T f - \int_{\R^{2d}} |\nabla_p f|^2 \, d\rho.
\end{equation}
\end{definition}

Define $L^2_\nabla (\rho)$ as the completion of $\{\nabla_p f: f\in C_c^\infty(\R^{2d})\}$ with respect to the norm
\[
\|\cdot\|_{\rho}^2  := \int_{\R^{2d}} |\cdot|^2\, d\rho.
\]
Note that, depending on $\rho$,  $\|\cdot\|_{\rho}$ may be only a seminorm and not a norm; but since the completion identifies elements that have zero distance in this seminorm, $L^2_\nabla(\rho)$ is a well-defined Hilbert space. Its elements are equivalence classes of measurable functions that are $\rho$-a.e. equal.
Also note that whenever $\mH(\rho)<\infty$, the function $(q,p)\mapsto p$ belongs to $L^2_\nabla(\rho)$.

The dual norm $\|\cdot\|_{-1,\rho}$ has an explicit representation:
\begin{lemma}
\label{lemma:charminusone}
\[
\|\T\|_{-1,\rho}^2 =
\begin{cases}
\displaystyle\int_{\R^{2d}} |h|^2 \, d\rho
& \text{if } \T = -\div_p (\rho h)  \text{ with } h  \in L^2_\nabla(\rho) ,\\
+\infty & \text{otherwise}.
\end{cases}
\]
\end{lemma}

\begin{proof}
Results of this type are common; this argument is adapted from~\cite{DG87}.

Since there is a one-to-one correspondence between $f\in C_c^\infty(\R^{2d})$ and $\nabla_p f\in L:=\{\nabla_p f: f\in C_c^\infty(\R^{2d})\}$, $\T$ can be considered to be a linear functional on $L$. If $\|\T\|_{-1,\rho}<\infty$, we can replace $f$ by $\lambda f$ and optimize with respect to $\lambda \in \R$ in~\eqref{def:minusonenorm}. We then find that
\[
\left|\dual \T f\right|\leq \|\T\|_{-1,\rho} \|\nabla_pf\|_{\rho}.
\]
Therefore $\T$ is  bounded with respect to the $L^2_\nabla(\rho)$-norm; it can be uniquely extended to a bounded linear functional on the whole of $L^2_\nabla(\rho)$, and  Riesz' representation theorem  implies the assertion of the Lemma.
\end{proof}

We can now state the large-deviation principle.
\begin{theorem}
\label{theo:LDP}
Assume that the initial data $(Q_i(0),P_i(0))$, $i=1,\dots,n$ are deterministic and chosen such that $\rho_n(0)\rightharpoonup \rho^0$ for some $\rho^0\in \P(\R^{2d})$. Then the empirical process $\{\rho_n\}$ satisfies a large-deviation principle in the space $C([0,T],\P(\R^{2d}))$, with good rate function
\begin{equation}
\label{def:I}
I(\rho) = \begin{cases}
\displaystyle \frac1{4\gamma \theta} \int_0^T \bigl\|\partial_t \rho_t - A_{\rho_t}^\tau \rho_t\bigr\|_{-1,\rho_t}^2 \, dt
&\text{if } \rho\in AC([0,T];\P(\R^{2d}))\text{ with }\rho|_{t=0}= \rho^0,\\
+\infty &\text{otherwise}.
\end{cases}
\end{equation}
The rate function $I$ can  also be written as
\begin{equation}
\label{rem:alternativeCharacterization}
I(\rho)=\begin{cases}
\displaystyle\frac1{4\gamma \theta} \int_0^T \int_{\R^{2d}}|h_t|^2 \,d\rho_t dt
\qquad & \text{if }\partial_t\rho_t=A_{\rho_t}^\tau\rho_t -\div_p(\rho_t h_t),
\text{ for }h\in L^2(0,T;L^2_\nabla(\rho_t)),\\
+\infty & \text{otherwise}.
\end{cases}
\end{equation}
%\begin{multline*}
%I(\rho) = \sup_{f\in C_c^\infty(\R^{2d}\times \R)}
%  \int_{\R^{2d}} f_T\,d\rho_T - \int_{\R^{2d}}f_0\,d\rho_0 - \int_0^T\int_{\R^{2d}} \big[A_{\rho_t} + \partial_t\big]f_t \, d\rho_t dt\\
%  - \frac12 \int_0^T\int_{\R^{2d}} |\nabla_p f_t|^2 \, d\rho_t dt.
%\end{multline*}
\end{theorem}
\begin{proof}
Setting $x= (q,p)$ and $b(x,\nu) = \bigl(p/m,-\nabla V(q) - (\nabla\psi*\nu)(q) - \gamma p/m\bigr)$ for $\nu\in\P(\R^{2d})$, the system~\eqref{eq:SDE} can be written as system of weakly interacting diffusions
\begin{equation}
\label{eq:weakly interacting system}
dX_i(t)=b(X_i(t),\rho_n(t))\, dt+ \sigma\,dW_i(t),
\end{equation}
where $W_i$ are  $d$-dimensional standard Wiener processes and for the length of this proof, $\sigma$ is the $2d\times d$ matrix
\[
\sigma = \sqrt{2\gamma\theta} \begin{pmatrix}0\\I_d\end{pmatrix}.
\]
Theorem 3.1 and Remark~3.2 of~\cite{BudhirajaDupuisFischer12} implies that $\rho_n$ satisfies a large-deviation principle with  rate function
\[
\widetilde I(\rho) := \inf \Expectation\left[\frac12 \int_0^T|U_t|^2 \, dt \right],
\]
where the infimum is taken over all processes $(\overline X, U, W)$ taking values in $\R^{2d}\times \R^d\times \R^{d}$ that solve
\begin{subequations}
\label{pb:theo:LDP}
\begin{align}
\label{pb:theo:LDP:1}
&d\overline X_t = b(\overline X_t,\rho_t) \, dt + \sigma U_t\, dt + \sigma \,dW_t,\\
&W \text{ is a standard $d$-dimensional Wiener process},\\
&\law \overline X_t = \rho_t \quad\text{for all }t.
\end{align}
\end{subequations}
For each such triple, for any  $f\in C_c^\infty(\R \times \R^{2d})$ the process
\[
 M_t :=  f_t(\overline X_t) - f_0(\overline X_0) - \int_0^t \bigl[(\partial_s + A_{\rho_s} +( \sigma U_s)\cdot \nabla)f_s\bigr](\overline X_s)\, ds
\]
is a martingale, and therefore $\Expectation M_t=\Expectation M_0=0$ for every $t>0$.

We now show~\eqref{def:I} by showing that $\widetilde I=I$. Define for any $\rho\in C([0,T];\P(\R^{2d}))$ and $f\in C_c^\infty(\R \times \R^{2d})$,
\[
J(\rho,f) := \int_{\R^{2d}}f_T\,d\rho_T - \int_{\R^{2d}} f_0\, d\rho_0
  - \int_0^T \int_{\R^{2d}} \bigl[(\partial_s + A_{\rho_s})f_s\bigr] d\rho_s ds - \gamma\theta \int_0^T \int_{\R^{2d}} |\nabla_p f_t|^2 \, d\rho_tdt.
\]
It is well known (see e.g.~\cite[Lemma~4.8]{DG87}) that
\[
I(\rho) = \sup_{f\in C_c^\infty(\R\times\R^{2d})} J(\rho,f).
\]
We have for any $f\in C_c^\infty(\R\times\R^{2d})$ and for any solution $(\overline X,U,W)$ of~\eqref{pb:theo:LDP},
\begin{align*}
\Expectation\left[\frac12 \int_0^T|U_t|^2 \, dt \right]
= \Expectation\left[\int_0^T\Bigl(U_t\nabla_p f_t(\overline X_t) - \frac12 |\nabla_pf_t(\overline X_t)|^2 \Bigr) \, dt \right]
+ \Expectation\left[\frac12 \int_0^T|U_t - \nabla_p f_t(\overline X_t)|^2 \, dt \right].
\end{align*}
Using $\Expectation M_T=0$ we rewrite this as
\begin{align}
\notag
&  \Expectation\left[ f_T(\overline X_T) - f_0(\overline X_0)
    - \int_0^T \bigl[(\partial_s + A_{\rho_s})f_s\bigr](\overline X_s)\, ds
    - \frac12 \int_0^T |\nabla_pf_s(\overline X_s)|^2  \, ds\right]\\
\notag
 &\qquad\qquad\qquad + \Expectation\left[\frac12 \int_0^T|U_t - \nabla_p f_t(\overline X_t)|^2 \, dt \right]\\
&= J\Bigl(\rho,\frac f{\sqrt{2\gamma\theta}}\Bigr) + \Expectation\left[\frac12 \int_0^T|U_t - \nabla_p f_t(\overline X_t)|^2 \, dt \right].
\label{eq:Exp-J}
\end{align}
Therefore
\[
\widetilde I(\rho) = \inf \Expectation\left[\frac12 \int_0^T|U_t|^2 \, dt \right]\geq \sup_f J(\rho,f)  = I(\rho).
\]

To prove the converse inequality, assume without loss of generality that $I(\rho)<\infty$. Using a reasoning similar to the proof of Lemma~\ref{lemma:charminusone} we find that there exists an $h\in L^2(0,T;L^2_\nabla(\rho_t))$ such that
\begin{equation}
\label{eq:FokkerPlanck-h}
\partial_t \rho_t -A^\tau_{\rho_t} \rho_t  = -\sqrt{2\gamma\theta} \div_p \rho_t h_t \qquad\text{in the sense of distributions.}
\end{equation}
Here the space $L^2(0,T;L^2_\nabla(\rho_t))$ is the Hilbert space obtained by closing $C_c^\infty(\R\times\R^{2d})$ with respect to the (semi-)norm
\begin{equation}
\label{norm:L2L2}
\|f\|_{\rho,T}^2 := \int_0^T \int_{\R^{2d}} |f(x,t)|^2 \, \rho_t(x)\, dt.
\end{equation}
We now construct a specific solution of~\eqref{pb:theo:LDP}. Let $(\widetilde X,W)$ be a solution of~\eqref{pb:theo:LDP:1} with $U=0$ and $\law \widetilde X_0 = \rho^0$; let $P$ be the law of $(\widetilde X,W)$ on $C([0,T];\R^{2d})\times C([0,T];\R^d)$.  Since $\|h\|_{\rho,T} < \infty$,
the process
\[
N_t := \sigma \int_0^t h_s(\widetilde X_s)\, dW_s
\]
is a $P$-square integrable continuous martingale with quadratic variation $\langle N\rangle_t = 2\gamma\theta t$.

Define $P_h$ as the modified law on $C([0,T];\R^{2d})\times C([0,T];\R^d)$ given by
\[
P_h := \exp\bigl[N_T  - \tfrac12 \langle N\rangle_T\bigr] \; P.
\]
By the Girsanov theorem (e.g.~\cite[Section~IV.4]{IkedaWatanabe81}) $P_h$ is the law of the unique solution $(X,W)$ of equation~\eqref{pb:theo:LDP:1} with $U_t = h_t(X_t)$, and since equation~\eqref{eq:FokkerPlanck-h} is the corresponding Fokker-Planck equation, it follows that the law of $X_t$ is equal to $\rho_t$. Therefore $(X,h\circ X,W)$ is a solution of~\eqref{pb:theo:LDP}.
Using~\eqref{eq:Exp-J} for this solution, we find for all $f$ that
\begin{align*}
\widetilde I(\rho) &\leq  J\Bigl(\rho,\frac f{2\gamma\theta}\Bigr)
+ \frac1{4\gamma\theta}\Expectation\left[ \int_0^T|h_t( X_t) - \nabla_p f_t( X_t)|^2 \, dt \right]\\
&\leq  I(\rho)
+ \frac1{4\gamma\theta}\Expectation\left[ \int_0^T|h_t( X_t) - \nabla_p f_t( X_t)|^2 \, dt \right]\\
&= I(\rho)  + \frac1{4\gamma\theta}\int_0^T \int_{\R^{2d}} |h_t(\xi)-\nabla_pf_t(\xi)|^2 \,\rho_t(d\xi) dt.
\end{align*}
Since $L^2(0,T;L^2_\nabla(\rho_t))$ is the closure of $C_c^\infty$ under the norm~\eqref{norm:L2L2},
\[
\inf_{f\in C_c^\infty(\R\times\R^{2d}) } \int_0^T \int_{\R^{2d}} |h_t(\xi)-\nabla_pf_t(\xi)|^2 \,\rho_t(d\xi) dt = 0.
\]
Hence $\widetilde I(\rho)\leq I(\rho)$ and this concludes the proof of~\eqref{def:I}. The form as in~\eqref{rem:alternativeCharacterization} of $I$ then follows from~\eqref{def:I} and Lemma~\ref{lemma:charminusone}.
\end{proof}

\begin{remark}
The structure of the large-deviation result of Theorem~\ref{theo:LDP} reflects a number of properties of the stochastic particle system~\eqref{eq:SDE}. To start with, the rate function is only finite if $\partial_t \rho - A_{\rho}^T\rho$ only has a perturbation in the $p$-direction, not in the $q$-direction; this reflects the fact in~\eqref{eq:SDE} that the noise is confined to the $P$-equation. In addition, the perturbation can only be in divergence form; this reflects the deterministic conservation of particles. Finally, the flux is of the form $\rho h$ where $h$ is in the closure $L^2_\nabla(\rho)$ of $p$-gradients; this property is also seen in the characterization of absolutely continuous curves in the Wasserstein metric~\cite[Theorem~8.3.2]{AGS08}.
\end{remark}

\begin{remark}
There is a large literature on large-deviation principles for stochastic particle systems; here we just mention a few results. Dawson and G\"artner~\cite{DG87} prove a large-deviations result for systems of interacting particles with non-degenerate diffusion, i.e., for nonsingular mobilities $\sigma$ with range $\R^{2d}$. Cattiaux and L\'eonard~\cite{CL94,CL94b} generalize the method of Dawson and G\"artner to singular mobilities, but for independent particles. In a separate paper~\cite{CattiauxLeonard95}, Cattiaux and L\'eonard also discuss the identification question treated in the proof of Theorem~\ref{theo:LDP} in more generality. Fischer~\cite{Fischer12TR} also proves identification results on related systems.

In the proof above we used the large-deviation result by Budhiraja \emph{et al.}~\cite{BudhirajaDupuisFischer12} above to obtain the large-deviation principle itself and a first characterization of the rate functional. The methods by which we identified $\widetilde I$ with $I$ are standard, but we did not find a theorem that suited our needs, and therefore we gave a separate proof.
\end{remark}

\medskip

For the sequel it will be useful to have a regularity result on the Hamiltonian $\mH$ (see~\eqref{def:VFPMeasureStructure.1}) associated with those curves $\rho$ for which  $I(\rho)$ is finite:
\begin{lemma}
\label{lemma:HinW12}
If $I(\rho)<\infty$ and $\mH(\rho_0)< \infty$, then the function $t\mapsto \mH(\rho_t)$ is an element of $W^{1,2}(0,T)$, and $\int_{\R^{2d}} p^2\, d\rho_t\in L^\infty(0,T)$.
\end{lemma}

\begin{proof}
By~\eqref{cond:Vpsi}, $\mH(\rho)$ bounds the integral $\int p^2/m^2 \, d\rho$ from above. Using the characterization of Remark~\ref{rem:alternativeCharacterization}, we formally calculate that
\begin{align}
\label{est:mH}
\partial_t \mH(\rho_t) &= \frac{\gamma \theta d}m - \gamma \int \frac{p^2}{m^2} \,d\rho_t
- \int \frac pm \cdot  h_t\, d\rho_t\\
&\leq \frac{\gamma \theta d}m - \gamma \int \frac{p^2}{m^2} \,d\rho_t
   + \gamma \int \frac{p^2}{m^2} \,d\rho_t+ \frac1{4\gamma} \int |h_t|^2 \, d\rho_t\notag\\
&= \frac{\gamma \theta d}m
  + \frac1{4\gamma} \int |h_t|^2 \, d\rho_t.
  \notag
\end{align}
This calculation can be made rigorous in its time-integrated form by approximating $p^2/m^2 + V(q)$ by a sequence of smooth functions $f_n\in C^\infty_c(\R^{2d})$, and using $f_n$ in the distributional form of the equation $\partial_t\rho_t=A_{\rho_t}^T\rho_t -\div_p(\rho_t h_t)$. Continuing with the proof, it follows that
\[
\sup_{t\in[0,T]} \mH(\rho_t) \leq \mH(\rho_0) + \frac{\gamma\theta d}m T + \frac1{4\gamma} \int_0^T \int |h_t|^2 \, d\rho_t = \mH(\rho_0) + \frac{\gamma \theta d}m T + \theta \,I(\rho) < \infty,
\]
and consequently $\int p^2\, d\rho_t$ is also uniformly bounded. We conclude by remarking that the right-hand side of~\eqref{est:mH}, as a function of time $t$, is an element of $L^2(0,T)$.
\end{proof}

\begin{remark}
\label{rem:lemma:HinW12:solution}
Note that a solution $\rho$ of~\eqref{eq:VFP} satisfies $I(\rho) = 0$, and therefore Lemma~\ref{lemma:HinW12} also applies to solutions of~\eqref{eq:VFP}.
\end{remark}

\section{Main results 2: The VFP equation and the large deviations in GENERIC form}

\label{sec:GENERICVFP}

In this section we reformulate both the VFP equation and the large-deviation rate functional of the previous section in terms of the GENERIC structure. It will become apparent that the large-deviation behaviour respects the GENERIC structure, in the sense that the rate function for this system can be formulated in an abstract form, using only the GENERIC building blocks. This will suggest in Section~\ref{sec:variational} a variational formulation for a very general GENERIC system.

\subsection{Making the VFP equation conserve energy}
As it stands, the  VFP equation~\eqref{eq:VFP} does not satisfy the conditions of GENERIC, since there is no conserved functional $\aE$. The reason for this is physical: the SDE~\eqref{eq:SDE} models a system of particles in interaction with a heat bath, and this interaction causes fluctuations of the natural energy (the Hamiltonian) of the particle system,
\begin{equation}
\label{def:Hn}
H_n(Q_1,\dots,Q_n,P_1,\dots,P_n) := \frac1n\sum_{i=1}^n \Big[\frac{P_i^2}{2m} + V(Q_i)\Big]
  + \frac1{2n^2}\sum_{\substack{i,j=1}}^n \psi(Q_i-Q_j).
\end{equation}
Indeed, combining~\eqref{eq:SDE} with It\^o's lemma the derivative of the expression above is
\[
-\frac1n\sum_{i=1}^n\left[\frac{\gamma}{m^2}P_i^2\, dt-\frac{\gamma \theta d}{m}\,dt+\frac{\sqrt{2\gamma \theta}}{m}P_i\,dW_i\right],
\]
which has no reason to vanish.
There is a simple remedy for this: we add a single scalar unknown~$e_n$ and define its evolution by the negative of the above, leading to the extended particle system
\begin{subequations}
\label{eq:SDE-extended}
\begin{align}
dQ_i&=\frac{P_i}{m}\,dt,
\\dP_i&=-\nabla V(Q_i)\,dt - \sum_{\substack{j=1}}^n \nabla \psi(Q_i-Q_j)-\frac{\gamma}{m}P_i\,dt+\sqrt{2\gamma \theta}\,dW_i,\\
\label{eq:SDE-e}
de_n &= \frac1n \sum_{i=1}^n\left[\frac{\gamma}{m^2}P_i^2\,dt -\frac{\gamma \theta d}{m}\,dt+\frac{\sqrt{2\gamma \theta }}{m}P_i\,dW_i\right],
\end{align}
\end{subequations}
with which  $H_n + e_n$ becomes deterministically constant. Note that $e_n$ can be interpreted as the energy of the heat bath; the flow of energy between the particle system and the heat bath is described by the flow of energy between $H_n$ and $e_n$.

Exactly the same arguments apply to the VFP equation~\eqref{eq:VFP}. At this level the analogue of the Hamiltonian $H_n$ is the functional $\mH$ defined in~\eqref{def:VFPMeasureStructure.1}, and indeed $\mH$ is not constant along a solution, as can be directly verified.
We mirror the arguments above and add a new variable $e$, \emph{depending only on time}, so that the solution space becomes $(\rho,e)\in \P(\R^{2d})\times \R$.
The full system is now defined by the VFP equation~\eqref{eq:VFP} plus the equation $de/dt = -(d/dt) \mH(\rho)$, that guarantees that $\mH(\rho)+e$ is conserved. When writing this equation in full, it becomes
\begin{subequations}
\label{pb:VFPExtended}
\begin{align}
\label{pb:VFPExtended:rho}
\partial_t \rho &= - \div_q\Bigl(\rho \frac pm\Bigr) + \div_p  \rho \Bigl(\nabla_q V + \nabla_q \psi * \rho + \gamma\frac pm \Bigr)+\gamma \theta \Delta_p \rho,\\[\jot]
\frac d{dt} e &= \gamma \int_{\R^{2d}} \frac{p^2}{m^2}\,\rho(dqdp) - \frac{\gamma \theta d}m.
\label{pb:VFPExtended:e}
\end{align}
\end{subequations}

We stress that this system is coupled only in one direction: the second equation is slaved to the first one. Note that equation~\eqref{pb:VFPExtended:e} is well-defined: if $\mH(\rho_0)<\infty$, then by Lemma~\ref{lemma:HinW12} and Remark~\ref{rem:lemma:HinW12:solution} $\mH(\rho_t)$ is bounded for all $t$;  therefore $\int p^2 d\rho_t$ is finite for all $t$.

By this simple mechanism a non-conserving system can be made conserving. Although mathematically this is no more than a trick, for this system it has physical meaning, as we argued above: the additional variable keeps track of the movement of energy between the particle system and the heat bath. We next show that the remaining conditions of GENERIC can also be verified.

\subsection{The VFP equation as a GENERIC system}
\label{subsec:VFPasGENERIC}

With the extension of the previous section, the VFP equation is formally a GENERIC system with the following building blocks:
\begin{equation}
\label{def:KramersInGENERIC}
\begin{aligned}
\aZ &= \P_2(\R^{2d}) \times \R,\quad & \aE(\rho,e) &= \mH(\rho) + e,
&\qquad \aL &= \aL(\rho,e) = \begin{pmatrix}\aL_{\rho\rho} & 0\\0 & 0\end{pmatrix},\\
\az& = (\rho,e),& \aS(\rho,e) &=
 \mS(\rho) + e,
&\qquad \aM &= \aM(\rho,e) = \gamma\begin{pmatrix}\aM_{\rho\rho}  & \aM_{\rho e}\\\aM_{e\rho} & \aM_{ee}\end{pmatrix},
\end{aligned}
\end{equation}
where the operators defining $\aL$ and $\aM$ are given, upon applying them to a vector $(\xi,r)$ at $(\rho,e)$, by
\begin{alignat*}3
\aL_{\rho\rho} \xi &= \div \rho \J \nabla \xi,&\qquad
\aM_{\rho\rho}\xi &= - \div_p \rho\nabla_p \xi, &\qquad \aM_{\rho e} r &= r\div_p \Big(\rho\frac pm\Big),\\
&&\aM_{e\rho}\xi  &= -\int_{\R^{2d}} \frac pm \cdot\nabla_p\xi \, \rho(dqdp)
&\qquad \aM_{ee} r &= r\int_{\R^{2d}} \frac{p^2}{m^2} \,\rho(dqdp).
\end{alignat*}
The space $\P_2(\R^{2d})$ is the subset of $\P(\R^{2d})$ with bounded second $p$-moments:
\[
\P_2(\R^{2d}) := \Bigl\{ \rho\in \P(\R^{2d}): \int_{\R^{2d}} p^2 \rho(dpdq) < \infty\Bigr\}.
\]
We equip $\P_2(\R^{2d})$ with the same weak topology as $\P(\R^{2d})$. Finally, the entropy $\mS$ is defined as
\[
\mS(\rho) := -\theta \int_{\R^{2d}} f(x)\log f(x) \, dx \quad \text{ whenever $\rho$ has Lebesgue density $f$.}
\]

With these definitions, equation~\eqref{eq:VFP} can be written as
\begin{equation}
\label{eq:GENERICVFP}
\partial_t\az_t = \aL(\az_t) \grad \aE(\az_t) + \aM(\az_t) \grad \aS(\az_t),
\end{equation}
where the gradient operators are to be interpreted as $L^2$-gradients.
At this stage, however, this equation is formal, since the sense in which this equation holds has not been specified. Rather than going into detail here, we defer this discussion to after the introduction of the variational structure in Section~\ref{sec:variational}.

The operators $\aL$ and $\aM$ can readily be seen to be antisymmetric and symmetric (with respect to the $L^2$-inner-product, since we use $L^2$-gradients as derivatives); for instance, in the case of $\aL$, we have for any vectors $(\xi_1,r_1)$ and $(\xi_2,r_2)$ at $(\rho,e)$ by partial integration that
\[
\dual{(\xi_1,r_1)}{\aL(\rho,e)(\xi_2,r_2)}
= \dual{\xi_1}{\aL_{\rho\rho}(\rho) \xi_2}
= \int_{\R^{2d}} \xi_1\div\rho\J\nabla \xi_2 =  -\int_{\R^{2d}} \nabla \xi_2\cdot \J^T\nabla \xi_1
\,\rho ,
\]
which is antisymmetric since $\J$ is antisymmetric (see~\eqref{def:symplecticMatrixJ}). The verification of the symmetry of $\aM$ is similar; the verification of the Jacobi identity~\eqref{def:Jacobi} is a tedious but elementary calculation, which  hinges on the fact that $\J$ is constant and antisymmetric. Finally, the verification of the degeneracy conditions~\eqref{def:degeneracy} is again straightforward.

\subsection{Large deviations for the VFP equation in GENERIC form}

We now reformulate the large-deviations rate functional of Theorem~\ref{theo:LDP} in terms of the GENERIC building blocks above, and therefore in terms of the extended unknown $\az= (\rho,e)\in \aZ$.  To do this, we also generalize the concepts of absolute continuity and introduce the appropriate norms.

\begin{definition}
The function $[0,T]\ni t \mapsto \az(t) = (\rho(t),e(t)) \in \aZ$ is \emph{absolutely continuous} if $\rho\in AC([0,T];\P_2(\R^{2d}))$ and $e\in AC([0,T];\R)$.
\end{definition}
\noindent
Again, if $\az$ is absolutely continous, then $\partial_t \az$ exists for almost all $t$ as an element of $\cD'(\R^{2d})\times \R$.

\medskip
The `matrix' $\aM$ generates a natural pair of semi-inner-products and seminorms.
\begin{definition}
Fix $\az = (\rho,e)\in \aZ$. The seminorms $\|\cdot\|_{\aM(\az)}$ and $\|\cdot\|_{\aM(\az)^{-1}}$ are defined as follows. For $(\xi,r)\in C_c^\infty(\R^{2d})\times \R$,
\begin{align*}
\|(\xi,r)\|_{\aM(\az)}^2  &:= \gamma \int_{\R^{2d}}\Bigl[
  \xi\aM_{\rho\rho}\xi + \xi \aM_{\rho e}r + r\aM_{e\rho}\xi + r\aM_{ee}r\Bigr]\, dx\\
&= \gamma \int_{\R^{2d}} \Bigl|\nabla_p \xi -r\frac pm\Bigr|^2 \, d\rho
= \gamma \left\|\nabla_p\xi -r\frac pm \right\|_{\rho}^2.
\end{align*}
For $(\T,s)\in \cD'(\R^{2d})\times \R$,
\begin{equation}
\label{def:aMminusone}
\|(\T,s)\|_{\aM(\az)^{-1}}^2 = \sup_{\substack{\xi\in C_c^\infty(\R^{2d})\\r\in\R}}
  2\dual \T\xi + 2sr - \|(\xi,r)\|_{\aM(\az)}^2.
\end{equation}
The inner products $(\cdot,\cdot)_{\aM}$ and $(\cdot,\cdot)_{\aM^{-1}}$ are then defined through the expression
$4(a,b) = \|a+b\|^2 - \|a-b\|^2$.
\end{definition}

As in the case of $L^2_\nabla(\rho)$, the $\aM$-seminorm is degenerate: there exist $\rho$, $\xi$, and $r$ for which it vanishes.
Let $\mathscr H_\aM$ be the set of equivalence classes of elements of $C_c^\infty(\R^{2d})\times \R$ with zero distance in this norm. On $\mathscr H_\aM$, the $\aM$-seminorm is a norm, and we define $H_\aM$ as the completion of $\mathscr H_\aM$ with respect to this norm. Note that $H_\aM$ can be identified with the space $L^2_\nabla(\rho)$, as follows. On one hand, if $(\eta_n, s_n)$ is a Cauchy sequence in $\mathscr H_\aM$, then
\[
\big\|(\eta_n,s_n)-(\eta_{n'},s_{n'})\big\|_{\aM} = \sqrt{\gamma}\Bigl\|\nabla_p (\eta_n-\eta_{n'}) -(s_n-s_{n'})\frac pm\Bigr\|_{\rho}\longrightarrow 0\quad \text{ as }n,n'\to\infty,
\]
so that $\nabla_p\eta_n -s_np/m$ is a Cauchy sequence in $L^2_\nabla(\rho)$ and thus converges to some $h\in L^2_\nabla(\rho)$; vice versa, for each $h\in L^2_\nabla(\rho)$ by definition there exists a sequence $\eta_n\in C_c^\infty$ such that $\nabla_p\eta_n\to h $ in $L^2_\nabla(\rho)$, and therefore $(\eta_n,0)$ is a Cauchy sequence in $\mathscr H_\aM$ corresponding to $h$.

Since the $\aM$-seminorm is degenerate, the $\aM^{-1}$-seminorm is singular. Indeed, Lemma~\ref{lemma:charminusone} implies the following
\begin{lemma}
Assume that $\int_{\R^{2d}} p^2 \, d\rho < \infty$. Then
\label{lemma:CharMminusone}
\[
\|(\T,s)\|_{\aM(\az)^{-1}}^2 =
\begin{cases}
\displaystyle\frac 1\gamma\int_{\R^{2d}} |h|^2 \, d\rho
& \text{if } \T = -\div_p \rho h  \displaystyle
  \text{ with } h \in L^2_\nabla (\rho) \text{ and } s =  - \int_{\R^{2d}}\frac pm \cdot h\, d\rho,\\
+\infty & \text{otherwise}.
\end{cases}
\]
\end{lemma}

\begin{proof}
As in the case of Lemma~\ref{lemma:charminusone}, $\|(\T,s)\|_{\aM(z)^{-1}}<\infty$ implies that $(\T,s)$ is a  linear functional on $C_c^\infty\times \R$, and by the assumption $\int p^2\,d\rho<\infty$ it is bounded with respect to the $\aM$-seminorm. Because of the identification with $L^2_\nabla(\rho)$ we can consider it as a bounded linear functional on $L^2_\nabla(\rho)$. By the Riesz representation theorem there exists an element $h\in L^2_\nabla(\rho)$ such that for all $\xi$ and $r$
\[
\dual{\T}\xi + rs =   \int_{\R^{2d}} h \Big(\nabla_p \xi - r\frac pm\Big)\, d\rho
=  \int_{\R^{2d}} h  \cdot \nabla_p \xi\, d\rho - \, r\int_{\R^{2d}} h \cdot \frac pm \, d\rho.
\]
From this identity the claim follows.
\end{proof}
\noindent

The rate function of Theorem~\ref{theo:LDP} now has a reformulation in terms of the objects that we have just defined.
\begin{lemma}
\label{theo:LDPGENERIC}
The rate function $I$ of Theorem~\ref{theo:LDP} can be written in terms of $\az$ as
\begin{equation}
\label{def:J}
J(\az) = \begin{cases}
\displaystyle
\int_0^T \frac1{4\theta} \bigl\|\partial_t\az_t - \aL(\az_t) \grad \aE(\az_t) - \aM(\az_t) \grad \aS(\az_t)\bigr\|_{\aM(\az_t)^{-1}}^2\, dt, \kern-2cm\\
 &\text{if }\az =(\rho,e) \in AC([0,T];\aZ) \text{ and } \rho_{t=0} = \rho^0,\\[2\jot]
+\infty  &\text{otherwise},
\end{cases}
\end{equation}
in the sense that
\[
J\big((\rho,e)\big)  = \begin{cases}
I(\rho) &\qquad\text{provided $t\mapsto \mH(\rho_t) + e_t$ is constant }\\
+\infty &\qquad \text{otherwise}.
\end{cases}
\]
\end{lemma}

\begin{proof}
First assume that $I(\rho)<\infty$. By~\eqref{def:I} and Lemma~\ref{lemma:charminusone} we have
\[
\partial_t \rho_t - A^\tau_{\rho_t}\rho_t = -\div \rho_t h_t,
\]
where $h\in L^2(0,T;L^2_\nabla(\rho_t))$. Define $e$ by
\[
e_0 := 0 \qquad \text{and}\qquad \partial_t e_t = \gamma \int_{\R^{2d}} \frac{p^2}{m^2}\,\rho_t(dqdp) - \frac{\gamma \theta d}m +\int_{\R^{2d}}\frac pm \, h_t \, \rho_t(dqdp).
\]
By Lemma~\ref{lemma:HinW12} the function $t\mapsto \int p^2\, d\rho_t$ is in $L^\infty(0,T)$, and since $h\in L^2(0,T;L^2_\nabla(\rho_t))$ the last term is in $L^1(0,T)$; therefore $e$ is well-defined, and an element of $AC([0,T];\R)$. By construction the function $t\mapsto \mH(\rho_t) + e_t$ is constant. Upon setting $\az := (\rho,e)$, an explicit calculation shows that $I(\rho)$ and $J(\az)$ are both equal to $(4\gamma\theta)^{-1}\int_0^T \int_{\R^{2d}} |h_t|^2 \,d\rho_tdt$.

A similar argument starts by assuming $J(\az)< \infty$ for $\az=(\rho,e)$ and showing that $I(\rho)$ and $J(\az)$ are again equal.
\end{proof}

\begin{remark}
Note how the condition of constant energy $\mH + e$ is contained in~\eqref{def:J} through the defintion of the seminorm $\|\cdot\|_{\aM^{-1}}$.
\end{remark}

\section{Main results 3: A variational formulation for GENERIC systems}
\label{sec:variational}

The functional $J$ in~\eqref{def:J} has the interesting property that it only depends on the GENERIC building blocks, and therefore makes sense, at least formally, for an arbitrary GENERIC system. We now explore the consequences of this observation for general GENERIC systems. The discussion in this section is therefore necessarily formal.

First, we note that the functional $J$ can be written in a different way by using one of the degeneracy conditions~\eqref{def:degeneracy}. As above, we associate a formal inner product with $\aM$ and $\aM^{-1}$ by
\[
(a,b)_\aM\colonequals {a}\cdot{\aM \,b}
\qquad\text{and}\qquad
(a,b)_{\aM^{-1}}\colonequals{a}\cdot {\aM^{-1} b}.
\]
(See Remark~\ref{rem:differentiability} for a discussion of the dot in these expressions).
Then the antisymmetry of $\aL$ and the first degeneracy condition in~\eqref{def:degeneracy} imply that
\[
\bigl(\aL \grad \aE,\aM\grad \aS\bigr)_{\aM^{-1}}
= {\aL \grad \aE}\cdot {\grad \aS}
= -{\grad \aE}\cdot {\aL \grad \aS} = 0.
\]
Therefore
\begin{align*}
\|\partial_t\az- \aL \grad \aE - \aM \grad \aS\bigr\|_{\aM^{-1}}^2
  &= \|\partial_t\az- \aL \grad \aE \bigr\|_{\aM^{-1}}^2+ \| \aM \grad \aS\bigr\|_{\aM^{-1}}^2
   + 2\bigl(\partial_t \az ,\aM \grad\aS\bigr)_{\aM^{-1}}\\
  &= \|\partial_t\az- \aL \grad \aE \bigr\|_{\aM^{-1}}^2+ \|  \grad \aS\bigr\|_{\aM}^2
   + 2 \,{\partial_t \az}\cdot {\grad\aS},
\end{align*}
so that
\begin{equation}
\label{def:J2}
2\theta J(\az) = \aS(\az(T)) - \aS(\az(0)) + \frac12 \int_0^T \Bigl[ \|\partial_t\az- \aL \grad \aE \bigr\|_{\aM^{-1}}^2+ \|  \grad \aS\bigr\|_{\aM}^2 \Bigr] \, dt.
\end{equation}

This discussion suggests a general variational formulation for any GENERIC system, as follows:
\begin{quote}
\textbf{Variational formulation of a GENERIC system: } Given a GENERIC system $\{\aZ,\aE,\aS,\aL,\aM\}$, define $J$ as in~\eqref{def:J2}. A function $\az\colon[0,T]\to \aZ$ is a solution of the GENERIC equation~\eqref{eq:GENERICeqn1} iff $J(\az)=0$.
\end{quote}
In full generality, this characterization is formal; no details about the functional setting are stated. In the example of the VFP equation, however, this formulation is exact, as described by Lemma~\ref{theo:LDPGENERIC}.

Indeed, let us now come back to the question in which sense the VFP equation satisfies the GENERIC equation~\eqref{eq:GENERICVFP}. The discussion above suggests that this variational formulation could be a natural solution concept. Indeed, for any $\az = (\rho,e)\in AC([0,T];\aZ)$ with finite $\aS(\az(0))$ each of the terms in~\eqref{def:J2} makes sense as an element of $(-\infty,\infty]$:
\begin{itemize}
\item $\aS(\az(T)) \in (-\infty,\infty]$ by definition;
\item The assumption that $\az\in AC([0,T];\aZ)$ implies that for almost all $t$, $\partial_t \rho $ is a distribution on $\R^{2d}$ and $\partial_t e$ exists in $\R$;
\item Under reasonable assumptions on $V$ and $\psi$, $\aL\grad \aE = -\div_q \rho p/m + \div_p \rho \big[\nabla_q V + \nabla_q \psi*\rho\big]$ is well-defined in the sense of distributions;
\item Therefore the seminorm $\|\partial_t \az - \aL\grad \aE \|_{\aM^{-1}}^2$ is well-defined in $[0,\infty]$;
\item The seminorm $\|\cdot\|_{\aM}^2$ can be assumed well-defined in $[0,\infty]$ for any argument, by extending it by $+\infty$ outside of $H_\aM$.
\end{itemize}

For the VFP equation there are  several other solution concepts that are natural for different reasons and have various advantages; examples are distributional solutions and solutions in the sense of semigroups (since the first and last terms on the right-hand side of~\eqref{eq:VFP}
form a hypoelliptic operator with a smooth and strictly positive fundamental solution). The relevance of this discussion therefore lies not so much in the specific case of the VFP equation, but more in the potential application to general GENERIC systems.

\begin{remark}
Gradient flows are GENERIC systems with $\aE=0$. For this class of systems, this variational formulation is well known and has been put to good use. For instance, Sandier and Serfarty~\cite{SandierSerfaty04} (see also e.g.~\cite{Serfaty09TR,Stefanelli08,Le08,AMPSV12})  showed how the variational form can be used to pass to limits in parameters in the equation. We expect something  similar might be possible for these GENERIC variational formulations, and will return to this in a future publication.
\end{remark}

\section{Synthesis}

Let us recapitulate what we have just seen.
\begin{itemize}
\item The VFP equation has a variational formulation of the type `$J(\az)\geq0$, and $J(\az) = 0 $ iff $\az$ is a solution';
\item This variational formulation, the functional $J$, is identical to the large-deviation rate functional for the stochastic particle system~\eqref{eq:SDE} for the case of fixed energy;
\item The equation and the variational formulation can both be written in terms of  only the GENERIC building blocks;
\item This suggests a variational formulation for an arbitrary GENERIC system.
\end{itemize}

In the remainder of this paper we discuss a number of consequences. In Section~\ref{sec:properties} we use the connection between the VFP equation, large deviations, and the GENERIC structure to shed some light on the properties of GENERIC as formulated in Section~\ref{subsec:GENERIC}. Section~\ref{sec:generalization} is devoted to the generalization mentioned in Section~\ref{subsec:VFP}.

\section{Interpretation of the GENERIC properties}
\label{sec:properties}

The GENERIC structure of the VFP equation,  introduced in Section~\ref{subsec:VFPasGENERIC}, does raise some questions.  Why are these bulding blocks the `right' ones, from a philosophical, or modelling point of view? Is it clear why $\aE$ and $\aS$ should be what they are defined to be in~\eqref{def:KramersInGENERIC}? Is it clear why $\aL$ and $\aM$ are what they are? Why they do indeed satisfy the various conditions described above?

In addition, the origin of the GENERIC properties themselves, as described in Section~\ref{subsec:GENERIC}, is somewhat obscure. Why should `every' thermodynamic system satisfy these properties?
We now show how the connection with large deviations of the underlying particle system gives us some answers to these questions.

%We first find out that microscopic~\eqref{eq:microL}-\eqref{eq:microS} and macroscopic~\eqref{eq:totalenergy}-\eqref{eq:frictionOP} GENERIC building blocks are connected together as mentioned in~\cite[Section III.B]{OG97}
%\[
%\left\langle\begin{pmatrix}\xi_1\\r_1\end{pmatrix},M\begin{pmatrix}\xi_2\\r_2\end{pmatrix}\right\rangle_{L^2\times \R}=\int_{\R^{2d}}\rho\sigma\sigma^T\left\langle\begin{pmatrix}\nabla\xi_1\\r_1\end{pmatrix},M_i\begin{pmatrix}\nabla\xi_2\\r_2\end{pmatrix}\right\rangle_{\R^{2d}\times \R}\,d\mathbf{x},
%\]
%\[
%\left\langle\begin{pmatrix}\xi_1\\r_1\end{pmatrix},L\begin{pmatrix}\xi_2\\r_2\end{pmatrix}\right\rangle_{L^2\times \R}=\int_{\R^{2d}}\rho\left\langle\begin{pmatrix}\nabla\xi_1\\r_1\end{pmatrix},L_i\begin{pmatrix}\nabla\xi_2\\r_2\end{pmatrix}\right\rangle_{\R^{2d}\times \R}\,d\mathbf{x},
%\]
%\[
%\E(\rho,e)=\int_{\R^{2d}} \rho E_i(\mathbf{x},e)\,d\mathbf{x},
%\]
%\[
%\S(\rho,e)=\int_{\R^{2d}}\rho S_i(\mathbf{x},e)\,d\mathbf{x}-\theta \int\rho\log\rho\,d\mathbf{x}.
%\]
%
%The underlying idea in Definition~\ref{def:GENERIC-diffusion} is the following~\cite[Section III.B]{OG97}. \eqref{eq:FK GENERIC-diffusion} can be written as a GENERIC
%\[
%\partial_tf=\widehat{L}\grad \widehat{E}+\widehat{M}\grad \widehat{S}
%\]
%where
%\begin{align*}
%\{A,B\}&=\langle\grad A,\widehat{L}\grad B\rangle=\int \nabla\grad A\cdot \widetilde{L}\nabla\grad B f,\\
%[A,B]&=\langle\grad A,\widehat{M}\grad B\rangle=\int \nabla\grad A\cdot \widetilde{M}\nabla\grad B f,\\
%\widehat{E}(f)&=\int \widetilde{E}f,\\
%\widehat{S}(f)&=\int [\widetilde{S}-\theta \log f]f.\\
%\end{align*}
%

\medskip

%\item \emph{The additive structure of the GENERIC}: The additive structure of the Kramers equation can be interpreted as consequence of \emph{the additive property of forces} and \emph{operator-splitting approximation scheme}. Firstly, as mentioned in the introduction, the system~\eqref{eq:SDE} describes the movement of a particle at position $Q$ and with momentum $P$ under the resulting effect of three forces: external force $-\nabla V$, friction force $-\frac{\gamma}{m}P$ and noise. The additive property of the forces results in the additive structure of the Kramers equation by using the Ito formula. Secondly, as shown in~\cite{DPZ12}, the Kramers equations can be obtained from an operator-splitting approximation scheme: a streaming followed by a minimization step. This, as an operator-splitting approximation scheme often does, leads to the additive structure of the Kramers equation.
%

\emph{The reversible operator\/ $\aL$ and the Hamiltonian $\mH$.} First consider the simpler case when $\psi=0$. Then the only non-zero component of the operator $\aL$, which is $\aL_{\rho\rho} = -\div\rho J\nabla$, is the Liouville operator for the Hamiltonian flow on $\R^{2d}$ generated by the symplectic matrix $\J$ and the Hamiltonian $H(q,p) = p^2/2m + V(q)$. Indeed, $x(t) = (q(t),p(t))$  solves the Hamiltonian equation
\begin{align*}
\frac d{dt} {x}=-\J\nabla H(x)
\end{align*}
      if and only if $\rho(t):=\delta_{{x}(t)}$ solves
\[
\partial_t \rho -\div (\rho\J\nabla H)=0.
\]
Therefore $\aL$ is the natural embedding of the symplectic geometry of $\J$ in $\R^{2d}$ into the space of measures $\P(\R^{2d})$; and when $\psi=0$,  $\mH(\delta_x) = H(x)$, and therefore $\mH$  similarly is the natural embedding of the $\R^{2d}$-space Hamiltonian $H$ into the space of measures.
The anti-symmetry and Jacobi identity properties of $L$ follow directly from that of the matrix~$\J$.

When $\psi$ is non-zero, a similar interpretation of $\mH$ is possible, since with the notation of~\eqref{def:Hn} we have
\[
\mH\big(\eta_n(x_1,\dots,x_n)\big) = H_n(x_1,\dots,x_n), \qquad\text{where}\qquad
\eta_n(x_1,\dots,x_n) := \frac1n \sum_{i=1}^n \delta_{x_i}.
\]
Similarly, $\aL$ can be interpreted as the embedding into $\P(\R^{2d})$ of the Hamiltonian flow on $\R^{2nd}$ generated by a symplectic matrix $\J_n$ consisting of $n$ copies of $\J$.

\medskip

\emph{The entropy functional\/ $\aS$.}
The functional $\aS$ in~\eqref{def:KramersInGENERIC} is defined as $ e+ \mS(\rho)  = e -\theta\int \rho\log \rho\,dx$. The second term in this sum is the usual entropy of $\rho$, multiplied by temperature $\theta$. Its form arises from the \emph{loss of information} in the mapping $\eta_n$ defined above. We explain it now for the case of finite state $S=\{1,\cdots, r\}$; the general case can be handled using the characterization of the relative entropy as a supremum over finite partitions~\cite[Lemma~1.4.3]{DupuisEllis97}. Let $X_1,\cdots, X_n$ be independent identically distributed $S$-valued random variables with common law $\mu$  on a probability space $(\Omega, \Sigma, \mathbb{P})$. Define the (random) empirical measure
\[
L_n:=\frac{1}{n}\sum_{i=1}^{n}\delta_{X_i}.
\]
There is a loss of information in going from $X_1,\cdots,X_n$ to the empirical measure $L_n$: $L_n(\omega)$ characterizes the observed frequencies of $\{1,\cdots,r\}$ among $X_1(\omega),\cdots, X_n(\omega)$, but does not tell us exactly what values  they take. The \emph{degree of degeneracy}, the number of possible ways that $X_1(\omega),\cdots, X_n(\omega)$ can be such that $L_n(\omega)$ is equal to a given $\rho=(\rho_i)_{i=1}^{n}=\left(\frac{k_1}{n},\cdots, \frac{k_r}{n}\right)$, where $(k_1,\cdots,k_r)\in \mathbb{N}^r$, $\sum_{i=1}^rk_i=n$, is
$
\frac{n\!}{k_1\!\cdots k_r\!}.
$
We have
\[
\Prob(L_n=\rho)=\frac{n\!}{k_1\!\cdots k_r\!}\prod_{i=1}^r\mu_i^{k_i},
\]
where $\mu_i=\mu({i})$ for $i=1,\cdots,r$. Hence
\[
\frac{1}{n}\log \Prob(L_n=\rho)=\frac{1}{n}\left(\log n\!-\sum_{i=1}^rk_i\!+\sum_{i=1}^rk_i\log\mu_i\right).
\]
Using  Stirling's formula in the form
\[
\log m\!=m\log m-m+o(m)\quad \text{as } m\rightarrow\infty,
\]
we find
\begin{align*}
\frac{1}{n}\log \mathbb{P}(L_n=\rho)&\approx\frac{1}{n}\left[ n\log n-n-\sum_{i=1}^r(k_i\log k_i-k_i)+\sum_{i=1}^rk_i\log \mu_i\right]
\\&=\log n-\sum_{i=1}^r\frac{k_i}{n}\log k_i+\sum_{i=1}^r\frac{k_i}{n}\log\mu_i \quad (\text{since}~ \sum_{i=1}^r k_i=n)
\\&=\sum_{i=1}^r\rho_i\left(\log n-\log k_i+\log \mu_i\right)\quad (\text{since}~ \rho_i=\frac{k_i}{n}~\text{and}~ \sum_{i=1}^r \rho_i=1)
\\&=\sum_{i=1}^r\rho_i\left(-\log \rho_i+\log \mu_i\right)
= -\sum_{i=1}^r \rho_i \log \frac{\rho_i}{\mu_i}.
\end{align*}
Retracing the steps in this computation we see that the term $\sum_{i=1}^r\rho_i\log\rho_i$ originates from the degree of degeneracy $
\frac{n\!}{k_1\!\cdots k_r\!}.
$

\medskip

\emph{The degeneracy condition $\aL\grad \aS=0$.} In the case of the VFP equation, this property holds true for any functional which depends locally on $\rho$, i.e., any functional of the form
    \[ F(\rho,e)=e+\int f(\rho)\,dx.
    \]
The functional $\aS$ indeed has this form with $f(\rho)=\rho\log\rho$. Therefore the degeneracy  $\aL \grad\aS=0$ holds exactly because the entropy is a local functional---and this locality is closely connected to the fact that the entropy characterizes the loss of information encountered when taking a limit and representing the system in terms of  (limits of) empirical measures, as described above.

\medskip

\emph{The irreversible operator $\aM$ and its properties.} To understand the operator $\aM$ we use an argument that we learned from Alexander Mielke. We transform the co-ordinates $\az=(\rho,e)$ to $\widetilde{\az}=(\widetilde{\rho},\widetilde{e})$, where
\[
    \widetilde{\rho}:=\rho,\qquad \widetilde{e}:=e+\int \grad \mH \, d\rho.
\]
%The functionals $\S$ and $\E$ become
%\[
%\widetilde{\E}=\widetilde{e}, \qquad\widetilde{\S}=-\int\widetilde{\rho}\log\widetilde{\rho}-\int\left(\frac{p^2}{2m}+V(q)\right)\widetilde{\rho}+\widetilde{e}.
%\]
Then the new variable $\widetilde \az$ again solves a GENERIC equation, with new building blocks $\widetilde \aL$, $\widetilde \aM$, $\widetilde \aE$, and $\widetilde \aS$.
Using the change-of-variable formula~\cite{OG97}, the operator $\widetilde\aM$ is given by
\begin{equation}
\label{eq: changeofvariableM}
%\widetilde{L}=\frac{\partial (\widetilde{x})}{\partial (x)}L\left[\frac{\partial (\widetilde{x})}{\partial (x)}\right]^T,
\widetilde{\aM}=\frac{\partial (\widetilde{\az})}{\partial (\az)}\aM\left[\frac{\partial (\widetilde{\az})}{\partial (\az)}\right]^T,
\end{equation}
where
\[
\frac{\partial (\widetilde{\az})}{\partial (\az)}=
\begin{pmatrix}
\frac{\partial \widetilde{\rho}}{\partial \rho}&\frac{\partial \widetilde{\rho}}{\partial e}\\
\frac{\partial \widetilde{e}}{\partial \rho}&\frac{\partial \widetilde{e}}{\partial e}
\end{pmatrix}
=\begin{pmatrix}
\id& 0\\
\int\square\grad \mH & \id
\end{pmatrix}
\]
is the transformation matrix. This formula should be read as operator composition; we write $\id$ for the identity operator, both for functions on $\R^{2d}$ and for elements of $\R$, and we use the notation
\[
\int\square\grad \mH
\qquad\text{for the operator} \qquad
\xi \mapsto \int\xi \grad \mH.
\]
 Hence
\begin{align*}
\widetilde{\aM}(\widetilde{\az})
&=\begin{pmatrix}
\id& 0\\
\int\square \grad \mH & \id
\end{pmatrix}
\begin{pmatrix}
-\div_p(\rho\nabla_p\square) & \square \div_p(\nabla_p\grad \mH) \\
- \int\nabla_p\grad \mH\cdot \nabla_p \square \,d\rho & \square \int|\nabla_p\grad \mH|^2\,d\rho
\end{pmatrix}
\begin{pmatrix}
\id& \square\grad \mH\\
0& \id
\end{pmatrix}
\\&=\begin{pmatrix}
\id& 0\\
\int\square\grad \mH & \id
\end{pmatrix}
\begin{pmatrix}
 -\div_p(\rho\nabla_p\square) & 0 \\
- \int\nabla_p\grad \mH\cdot \nabla_p \square \,d\rho & 0
\end{pmatrix}
\\&=
\begin{pmatrix}
-\div_p(\rho\nabla_p\square)  & 0 \\
0 & 0
\end{pmatrix}.
\end{align*}
%\begin{align*}
%\widetilde{L}(\widetilde{x})
%&=\begin{pmatrix}
%\mathbb{I}& 0\\
%\int\left(\frac{p^2}{2m}+V(q)\right)\square & 1
%\end{pmatrix}
%\begin{pmatrix}
%\div_{qp}\widetilde{\rho} \mathbf{J} \nabla_{qp}\square&0\\
%0&0
%\end{pmatrix}
%\begin{pmatrix}
%\mathbb{I}& \frac{p^2}{2m}+V(q)\\
%0& 1
%\end{pmatrix}
%\\&=\begin{pmatrix}
%\mathbb{I}& 0\\
%\int\left(\frac{p^2}{2m}+V(q)\right)\square & 1
%\end{pmatrix}
%\begin{pmatrix}
%\div_{qp}\widetilde{\rho}\mathbf{J}\nabla_{qp}\square & \div_{qp}\widetilde{\rho}\mathbf{J}\nabla_{qp}\left(\frac{p^2}{2m}+V(q)\right)\square \\
%0 & 0
%\end{pmatrix}
%\\&=
%\begin{pmatrix}
%\div_{qp}\widetilde{\rho}\mathbf{J}\nabla_{qp}\square & \div_{qp}\widetilde{\rho}\mathbf{J}\nabla_p\left(\frac{p^2}{2m}+V(q)\right)\square \\
% \int\left(\frac{p^2}{2m}+V(q)\right)\div_{qp}\widetilde{\rho}\mathbf{J}\nabla_{qp}\square & \int\left(\frac{p^2}{2m}+V(q)\right)\div_{qp}\widetilde{\rho}\mathbf{J}\nabla_{qp}\left(\frac{p^2}{2m}+V(q)\right)\square
%\end{pmatrix}
%\end{align*}
%
%and

These remarks now enable us to comment on the form of $\aM$.
First, the transformation to a different set of variables has the effect of `cleaning up' the operator $\aM$: in the new variables $\widetilde \az$, the operator only acts on the $\rho$ variable. Also, The operator $\widetilde{\aM}$ is clearly symmetric and positive semi-definite. The same properties for $\aM$ then follow as a consequence of~\eqref{eq: changeofvariableM}.

The operator $-\div_p(\rho\nabla_p\square)$ that appears in $\widetilde \aM$ is a familiar figure. It also appears in the characterization of Wasserstein gradient flows~\cite{ADPZ12}, and originates in the fluctuation behaviour of the Brownian noise in those systems---as is the case in Theorem~\ref{theo:LDP}. In the SDE~\eqref{eq:SDE}, however,  the noise only appears in the $P$-variable, and as a consequence the operator $-\div_p(\rho\nabla_p\square)$ also only operates on the $p$-variables. The symmetry of this operator is a consequence of It\=o's formula: in this formula for the stochastic evolution of functions $f(X_t)$ of a stochastic variable $X_t$, the second derivative $d^2f$ appears, and this second derivative gives rise to the second-order derivative in~$-\div_p(\rho\nabla_p\square)$. The symmetry of this expression therefore has the same origin as the symmetry of second-derivative matrices of functions.

In the new variables, the degeneracy condition $\widetilde \aM \grad \widetilde \aE$ is natural; indeed, $\widetilde \aE(\widetilde \az) = \widetilde \aE\bigl((\widetilde \rho,\widetilde e)\bigr) = \widetilde e$. Therefore $\mathop{\widetilde \grad} \widetilde \aE = (0,1)$, and the degeneracy condition coincides with the property that only $\widetilde \aM_{\rho\rho}$ is non-zero.

\bigskip

To conclude, the connection between large deviations and the GENERIC structure in the case of the VFP equation allows us to understand and explain where the various properties of the GENERIC formalism come from:
\begin{itemize}
\item The antisymmetry and the Jacobi identity of $\aL$ follow from the same properties of the underlying Hamiltonian system;
\item The symmetry of $\aM$ follows from the symmetry of second derivatives, as they appear in It\=o's formula;
\item The energy $\aE$ is (an extended version of) the Hamiltonian of the underlying system, after embedding into the space of measures;
\item The entropy $\aS$ characterizes the loss of information upon passing to empirical measures, in the sense of large deviations;
\item The degeneracy condition $\aL\grad \aS = 0$ arises from the fact that $\aS$ is a local functional;
\item The degeneracy condition $\aM\grad \aE = 0$ arises as a consequence of energy conservation.
\end{itemize}

\section{GENERIC formulation of the generalized VFP equation and its variational structure}

\label{sec:generalization}

Once the variational  structure of the VFP equation~\eqref{eq:VFP} has been recognized, a natural generalization of the VFP equation presents itself. By replacing the various terms by their equivalents in terms of $\mS$ and $\mH$ one arrives at equation~\eqref{eq:generalizedKramers}.
In this section, we show that this equation, after extension, also is a GENERIC system for abitrary $\mS$ and $\mH$, and we compute the corresponding functional $J$ explicitly. This section is  necessarily formal.

\medskip

By computing the derivative $\partial \mH(\rho_t)$ for a solution~$\rho$ of~\eqref{eq:generalizedKramers} we construct the extended version of~\eqref{eq:generalizedKramers}:
\begin{subequations}
\label{pb:gKRe}
\begin{align}
\partial_t\rho&=\div(\rho \J\nabla \grad \mH) +\div\bigl(\D(\rho)\nabla \grad (\mH+\mS)\bigr),\\
\frac d{dt} e &= \int_{\R^{2d}} \nabla \grad\mH \cdot \D(\rho) \cdot \nabla \grad (\mH + \mS).
\end{align}
\end{subequations}
Here $\D(\rho) := \rho\sigma\sigma^T$.
The corresponding GENERIC building blocks are
\begin{equation}
\label{def:gKReInGENERIC}
\begin{aligned}
\aZ &= \P_2(\R^{2d}) \times \R,\quad & \aE(\rho,e) &= \mH(\rho) + e,
&\qquad \aL &= \aL(\rho,e) = \begin{pmatrix}\aL_{\rho\rho} & 0\\0 & 0\end{pmatrix},\\
\az& = (\rho,e),& \aS(\rho,e) &= \mS(\rho) + e,
&\qquad \aM &= \aM(\rho,e) = \gamma\begin{pmatrix}\aM_{\rho\rho}  & \aM_{\rho e}\\\aM_{e\rho} & \aM_{ee}\end{pmatrix},
\end{aligned}
\end{equation}
where the components of $\aL$ and $\aM$ are given by
\begin{alignat*}3
\aL_{\rho\rho} \xi &= \div \rho \J \nabla \xi,&\qquad
\aM_{\rho\rho}\xi &= - \div \bigl(\D(\rho) \nabla \xi\big), &\qquad \aM_{\rho e} r &= r\div \bigl(\D(\rho)\nabla \grad \mH\big),\\
&&\aM_{e\rho}\xi  &= -\int_{\R^{2d}} \nabla\xi^T \cdot\D(\rho)\cdot  \nabla \grad \mH
&\qquad \aM_{ee} r &= r\int_{\R^{2d}}(\nabla \grad \mH )^T \cdot \D(\rho)\cdot \nabla \grad \mH.
\end{alignat*}

Most of the GENERIC properties of Section~\ref{subsec:GENERIC} follow immediately from this setup, such as the antisymmetry and symmetry of $\aL$ and $\aM$, the Jacobi identity, the positive semidefiniteness of~$\aM$. The degeneracy condition $\aM\grad \aE=0$ can be checked explicitly, but it can also be understood in the same way as in Section~\ref{sec:properties}, by first transforming the system to a new set of variables.

Finally, the degeneracy condition $\aL \grad \aS$ requires a specific assumption, as we already encountered above:

\begin{lemma}
If $\aS(\rho) = \int f(\rho)$ for some function $f$, then the system~\eqref{pb:gKRe} is a GENERIC system with the building blocks~\eqref{def:gKReInGENERIC}.
\end{lemma}
\noindent
The proof consists of simple verification.

\medskip
By following the same arguments as in Section~\ref{sec:LDPVFP}, we find a variational formulation of exactly the same type: a curve $\az\in AC([0,T];\aZ)$ is a variational solution if $J(\az)=0$, where $J$ is defined by~\eqref{def:J2} with building blocks~\eqref{def:gKReInGENERIC}. We have the following characterization:
\begin{lemma}
%
%
%
%
%From now on, we define the right hand side of~\eqref{eq:generalizedKramers} and~\eqref{eq:gKRe} respectively by $\mathrm{K}(\rho)$ and $\mathrm{Ke}(\rho,e)$.  \red{Different names?}
%
%\subsection{Variational formulation}
%\label{sec:Var gKRe}
%Now we find the explicit expression for the functional $J(\rho,e)$ defined as in~\eqref{eq:Jfunctional} for the GENERIC evolution in Theorem~\ref{theo:GENERIC-gKramers}.
%\begin{theorem}
For equation~\eqref{pb:gKRe} the functional $J$, defined in~\eqref{def:J2}, can be characterized as follows: If
\[
\frac{d}{dt}\begin{pmatrix}
\rho\\
e
\end{pmatrix}=\mathrm{VFPg}(\rho,e)+\begin{pmatrix}
\div (\D(\rho)\nabla \eta)\\
\int \D(\rho)\nabla \eta\cdot\nabla \grad\mH
\end{pmatrix},
\]
then
\[
J(\rho,e)=\frac{1}{2}\int_0^T\int_{\R^{2d}}\nabla \eta^T\cdot\D(\rho)\cdot\nabla \eta\,d{x}\,dt.
\]
Here $\mathrm{VFPg}(\rho,e)$ is the right-hand side of~\eqref{pb:gKRe}.
\end{lemma}
\noindent
The proof follows the same lines as as Lemmas~\ref{lemma:CharMminusone} and~\ref{theo:LDPGENERIC}.

\begin{center}
\textbf{Acknowledgement}
\end{center}
We gratefully acknowledge helpful discuss with Alexander Mielke, Peter M{\"o}rters and Frank Redig. The research of the paper has received funding from the ITN ``FIRST" of the Seventh Framework Programme of the European Community (grant agreement number 238702).
\bibliography{KramersGENERIC}
\bibliographystyle{alpha}
\end{document}